\documentclass[12pt]{article}
\usepackage{amssymb}
\usepackage{amsthm}
\usepackage{graphicx}
\usepackage{graphics}
\usepackage{amsthm}

\input pictex.tex
\newcommand{\clo}{\mathrm{S}^1}

\theoremstyle{definition}

\newtheorem{thm}{Theorem}[section]

\newtheorem{prop}[thm]{Proposition}

\newtheorem{lem}[thm]{Lemma}

\newtheorem{rem}[thm]{Remark}

\newtheorem{qs}[thm]{Question}

\input epsf

\begin{document}

\date{}
\author{Andr\'es Navas}

\title{Some examples of affine isometries of Banach spaces 
arising from 1-D dynamics}
\maketitle

\noindent{\bf Abstract.} We provide a large family of examples of affine isometries of the Banach spaces $C^0 (\clo)$, $L^1 (\clo)$ and $L^2 (\clo \times \clo)$ 
that are fixed-point-free despite being recurrent (in particular, they have zero drift). These come from natural cocycles on the group of circle diffeomorphisms, 
namely the logarithmic, affine and (a variation of the) Schwarzian derivative. Quite interestingly, they arise from diffeomorphisms that are generic in 
an appropriate context. We also show how to promote these examples in order to obtain families of commuting isometries satisfying the same properties.

\vspace{0.2cm}


\vspace{0.2cm}

\noindent{\bf MCS 2020:} 22F99, 37C20, 37E10, 46B04, 51F99.

\vspace{0.75cm}

A single (affine) isometry\footnote{To simplify the discussion, here we call an {\em isometry} what some 
authors call an {\em isometric embedding}, which is nothing but an isometric bijection of the space into itself. However, the discussion 
in \S \ref{s:ejemplos} still applies to {\em isometric maps}, i.e. to maps that preserve distance but are not necessarily bijective.}  
$I$ of a finite-dimensional Banach space $\mathbb{B}$ has a very simple dynamics: either there is a fixed point, or there 
is an ``axis of translation''; in the latter case, $I^n (v) / n$ converges as $n \to \pm \infty$ to a nonzero vector $c^*$ for every $v \in \mathbb{B}$. 
(In \S \ref{s:ejemplos}, we give a very short proof of this very well-known fact that is inspired from Riesz' proof of the mean ergodic theorem; see 
also \cite{Ed'64} for a proof using basic linear algebra.) A nice consequence of this is that, in finite dimension, an isometry $I$ cannot have zero 
drift without having a fixed point. Recall that the {\em drift} of $I$ is the value of the limit of $\| I^n (v) \| / n = \mathrm{dist} (I^n (v), 0) / n$. Such 
a limit exists for isometries of arbitrary metric spaces (by subadditivity)  and does not depend on $v$ (by the triangle inequality). 

This dichotomy ``fixed point\,/\,nonzero drift'' is far from being true in the infinite dimensional case, even for Hilbert spaces. This very old 
remark (going back to Carleman and Riesz~?) was cleverly pointed out in a very strong form by Edelstein in \cite{Ed64}. Namely, 
he produced examples of fixed-point-free affine isometries of $\ell^2$ with zero drift that are {\em recurrent}, in the sense that 
there exists an increasing sequence $r_n$ such that $I^{r_n}$ converges to the identity in the strong topology as $n \to \infty$. 
(This means that $I^{r_n} (v)$ converges to $v$ for each $v$ in the space.) 
See \S \ref{s:ejemplos} for a discussion and slight extension of Edelstein's examples. 

The purpose of this Note is to show that affine isometries with these ``strange'' properties arise in an abundant way in the context of classical smooth 
dynamical systems, more precisely, associated to circle diffeomorphisms. In \S \ref{s:circle}, we review the necessary background on these dynamics. 
In \S \ref{s:banachs}, we give natural examples of recurrent, fixed-point-free affine isometries of the Banach spaces $C^0 (\clo)$, $L^1 (\clo)$ and 
$L^2 (\clo \times \clo)$.  As we will see, these turn to arise ``generically'' in the closure of the set of $C^{\infty}$ circle diffeomorphisms with irrational 
rotation number. In particular, they are associated to a dense sets of diffeomorphisms.  This motivates the following question (compare \cite{Ei10}).

\vspace{0.32cm}

\noindent{\bf Question 1.} {\em Given an isometry \, $I = \Theta + c$ \, of an infinite dimensional Hilbert space with zero drift,  
%such that $1 \notin \mathrm{spec} (\Theta)$, 
can one approximate $I$ (in some natural topology) by isometries $I_n = \Theta_n + c_n$ that are 
recurrent but fixed-point-free~? Can such a perturbation be made ``generic'' in an appropriate framework~?}

\vspace{0.32cm}

It is very likely that isometries sharing a similar property can be constructed starting from diffeomorphisms of higher dimensional tori, 
particularly pseudo-rotations, given the very nice recent advances in their study (see e.g. \cite{WZ18}). However, we do not pursue in this 
direction since it seems that, from the functional-analytic perspective, this wouldn't produce something particularly different from what is already 
obtained using circle diffeomorphisms. More importantly, we see no way of extending our most interesting construction, namely that in $L^2$, 
since the cocycle we use is only suitable for the one-dimensional setting. Needless to say, to build our examples we do not use 
particularly new results concerning circle diffeomorphisms, but just the classical ones from \cite{He79} (even those of \cite{Yo84} are 
not strictly necessary).

In \S \ref{s:com}, we show how to promote the examples built in \S \ref{s:banachs} to obtain families of commuting isometries satisfying 
similar properties. More precisely, on each of the Banach spaces listed above, we give examples of groups of affine isometries isomorphic 
to $\mathbb{Z}^{\infty}$ such that every nontrivial element is recurrent and fixed-point-free. (We explain in \S \ref{s:ejemplos} how to obtain 
such a group of affine isometries of $\ell^2$ using Edelstein's construction.) Inspired by \cite{CTV07,CTV08}, we state a question in this 
regard that seems interesting from a group theoretical perspective.

\vspace{0.32cm}

\noindent{\bf Question 2.} {\em What are the (say, finitely generated) groups that may be realized as groups of affine isometries 
of a Hilbert space so that every infinite order element is recurrent but fixed-point-free~?}

\vspace{0.32cm}

Obvious examples of groups not having this property are nontorsion groups with Kazhdan's property (T) and, more generally, groups 
with the relative property (T) with respect to a nontorsion subgroup (see e.g. \cite{Na02, Na05} and references therein). However, 
there should certainly be more examples. In the opposite direction, it is very natural to ask whether nilpotent groups do have this property. 

\vspace{0.35cm}

Part of the content of this Note is spread as remarks/exercises in several sources, as for instance \cite{Na11} (see Exercise 5.2.26 therein) and \cite{Na23}. 
However, they appear with only few details and no major developments (in particular, no claim concerning commuting diffeomorphisms is made in these 
works). This text arose from the necessity of giving a concise yet clarifying discussion and a slight extension of all of this. Besides, it allows giving still another 
illustration of how useful the methods of dynamical systems can be the subject, particularly because of the flexibility of several well-developed perturbation 
techniques aiming to produce maps (e.g. manifold diffeomorphisms) with strange (chaotic) dynamics that are actually generic in appropriate contexts. This is 
somehow reminiscent of the celebrated Alspach's example \cite{Al81} of an isometry of a weakly compact convex subset of an $L^1$ space with no fixed point, 
which is nothing but a clever functional analytic reinterpretation of the classical baker map. Indeed, this example has many natural variations/generalizations arising 
from ergodic theory (see for instance \cite{Sc82}). In this spirit, motivated by \cite{Ka24}, it would be interesting to study the case of positive drift aiming to 
find the invariant ``functionals'' (Busemann like functions) in each framework discussed below (particularly the $L^2$ one).

I strongly thank Anders Karlsson for useful discussions that motivated the elaboration of this Note, as well as H\'el\`ene Eynard-Bontemps, Mario Ponce and 
Ignacio Vergara for his interest on it. I also thank Yves de Cornulier and Alain Valette for (somewhat old) discussions on the third (i.e. the $L^2$) construction 
presented here, as well as Tanja Eisner and Aimo Hinkkanen for a couple of valuable references and remarks.

%%%%%%%%%%%%%%%%%%%%%%%%%%%%%%%%%%%%%%%%%%%%%%%%%%%%%%%%%%%%%%%%%%%%%%%%%%%%%%%%%%%%%

\section{Some general remarks on affine isometries}
\label{s:ejemplos}

Let us consider an isometry $I := \Theta + c$ of the Euclidean space $\mathbb{R}^n$, where $\Theta \in O(n)$ is an orthogonal 
linear map and $c \in \mathbb{R}^n$ is a cocycle. For all $n \in \mathbb{N}$ and all $v \in \mathbb{R}^n$, we have
\begin{equation}\label{eq:iteraciones} 
I^n (v) = \Theta^n (v) + \sum_{i=0}^{n-1} \Theta^i (c); 
\qquad 
I^{-n} (v) = \Theta^{-n} (v) - \sum_{i=1}^n \Theta^{-i} (c). 
\end{equation}
Let $Im := (\Theta - Id) (\mathbb{R}^n)$. Recall that $Im^{\perp}$ is the subspace $Fix$ of fixed points of $\Theta$, since 
\begin{eqnarray*}
u \in Im^{\perp} 
&\iff& \langle \Theta (u')- u', u \rangle = 0 \mbox{ for all } u' \in \mathbb{R}^n \\
&\iff&  \langle u', \Theta^{-1}(u) - u \rangle = 0 \mbox{ for all } u' \in \mathbb{R}^n 
\,\, \iff \,\, \Theta^{-1}(u) = u 
\,\, \iff \,\, u = \Theta (u).
\end{eqnarray*}
Let hence $c = \bar{c} + c^*$ the decomposition of $c$ with $\bar{c} \in Im$ and $c^* \in Fix$. Then (\ref{eq:iteraciones}) 
becomes, for all $n \geq 1$,  
$$I^n (v) = \Theta^n (v) +  \sum_{i=0}^{n-1} \Theta^i (\bar{c}) + nc^* ; 
\qquad 
I^{-n} (v) = \Theta^{-n} (v) -  \sum_{i=1}^{n} \Theta^{-i} (\bar{c}) - nc^* .
$$
If we take $w$ such that $\bar{c} = \Theta (w) - w$, then the previous equalities become, via a telescopic sum trick, 
$$I^n (v) = \Theta^n (v) +  (\Theta^n (w)-w) + nc^*; 
\qquad 
I^{-n} (v) = \Theta^{-n} (v) - \Theta^{-1}(\Theta^{-n} (w)-w)  - n c^*; 
$$
Thus, 
$$\Big\| \frac{I^n(v)}{n} - c^* \Big\| = \frac{ \big\| \Theta^n (v) +  (\Theta^n (w)-w) \big\| }{n} \leq \frac{\|v\| + 2 \|w\|}{n} \longrightarrow 0$$
and
$$\Big\| \frac{I^{-n}(v)}{-n} - c^* \Big\|= \frac{ \big\| \Theta^{-n} (v) - \Theta^{-1}  (\Theta^{-n} (w)-w) \big\|}{n} \leq \frac{\|v\| + 2 \|w\|}{n} \longrightarrow 0$$
as $n \to \infty$. Finally, if $c^* = 0$, then $\{ I^n (v) \!: n \in \mathbb{Z} \}$ is a bounded set, and therefore $I$ has a fixed point (namely, the Chebyshev 
center of this set \cite{CNP13}). Note that, in the case $c^* \neq 0$, the translation $v \to v+w$ conjugates $I$ to the isometry $\hat{I} := \Theta + c^*$, for which the 
iterates take the simple form $\hat{I}^n (v) = \Theta^n (v) + nc^*. $ 

The computation above works more generally for any (nondegenerate) scalar product on $\mathbb{R}^n$. Now, if $\Theta$ is a linear isometry for another metric, 
it is not hard to produce a scalar product that is invariant by $\Theta$, so that the previous argument applies. One way to produce this invariant scalar product 
consists in noting that the set $S$ of maps $\{ \Theta^n \!: n \in \mathbb{Z} \}$ is precompact for the norm topology (this readily follows by looking at the set of 
the restriction maps to the unit ball). One can thus take the mean along the orbit by the (compact) closure of $S$ of any given scalar product (computed with 
respect to the Haar measure). Alternatively, to avoid using an invariant scalar product, one can use a theorem from \cite{Na13} to produce an explicit fixed point, 
namely the barycenter of the closure of any orbit with respect to the Haar measure on the (compact) closure of the group $\{I^n \!: n \in \mathbb{Z}\}$.

\begin{rem} The existence of a limit for $I^n (v) / n$ holds more generally for contractions of Hilbert spaces \cite{Pa71}, 
and even on reflexive Banach spaces\footnote{A beautiful alternative proof of this fact, due to Karlsson, appears in  
\cite{Ka}; see also \cite{Go18} for a very concise discussion on his beautiful argument.} \cite{KN81} and slightly more general normed spaces 
\cite{KN'81}. It is not true for isometric maps on general Banach spaces, as it is shown by the following example (with drift 1): 
$$(x_1,x_2,\ldots) \in \ell^1 \, \mapsto \, (1,x_2,x_2,\ldots) \in \ell^1$$
There seems to be no example in the literature of a Banach space isometry for which the limit does not exist,  
and the situation is unclear to us. 
\end{rem}

Let us next revisit Edelstein's construction using some simplifying arguments from \cite{Va07}. 
We consider the map $I$ of the (complex) Hilbert space $\ell^2$ given by 
\, $(x_1,x_2,\ldots) \mapsto (y_1, y_2, \ldots)$, \, where  \, 
$y_n = e^{\frac{2 \pi i}{n!}} (x_n - 1) + 1 = e^{\frac{2 \pi i}{n!}} x_n + \big( 1 - e^{\frac{2 \pi i}{n!}} \big)$.

\vspace{0.25cm}

\noindent{\bf I maps $\ell^2$ into $\ell^2$:} Note that $( 1 - e^{\frac{2\pi i}{n!}} )_n$ belongs to $\ell^2$, since, for $n \geq 2$,  
$$\ \big\| 1 - e^{\frac{2 \pi i}{n!}} \big\| 
= \left( 2 \Big( 1 - \cos \big(\frac{2\pi}{n!} \big) \Big) \right)^{1/2} 
\leq 2  \sin \big(\frac{2\pi}{n!} \big) 
\leq \frac{4\pi}{n!}.
$$
Since $\| e^{\frac{2 \pi i}{n!}} x_n \| = \| x_n \|$, this yields 
$$\|(y_1,y_2,\ldots)\|_{\ell^2} \leq \|(x_1,x_2,\ldots)\|_{\ell^2} +  \left( \sum_{n\geq 4} \Big( \frac{4\pi}{n!} \Big)^2\right)^{1/2}.$$
%+ 2 \|x_n\| \Big(1 - \cos \big( \frac{2\pi}{n!} \big) \Big) + \Big(1 - \cos \big( \frac{2\pi}{n!} \big) \Big)^2.$$
%Using that \, $1 - \cos \big( \frac{2\pi}{n!} \big) \leq \sin^2 \big( \frac{2\pi}{n!} \big)$ \, for $n \geq 3$, we conclude that 
%$\| I (x_1,x_2,\ldots) \|_{\ell^2}^2 = \sum_n \| y_n \|^2$  is less than or equal to 
%$$2 \|(x_1,x_2,\ldots)\|_{\ell^2}^2 + 2 \| (x_1,x_2,\ldots) \|_{\ell^2} \Big(\sum_n \sin^4 \big( \frac{2\pi}{n!} \big) \Big)^{1/2} 
%+ \sum_n \sin^4 \big( \frac{2\pi}{n!} \big) \,\, < \,\, \infty.$$

\vspace{0.25cm}

\noindent{\bf I is an isometry:} Letting $(y_1',y_2',\ldots) := I (x_1',x_2',\ldots)$, one obtains
$$\| y_n - y_n' \| = \| e^{\frac{2\pi i}{n!}} (x_n-x_n') \| = \| x_n - x_n' \|,$$
hence $\| I(x_1,x_2,\ldots) - I(x_1',x_2',\ldots) \|_{\ell^2} = \| (x_1,x_2,\ldots) - (x_1',x_2',\ldots) \|_{\ell^2}$. Moreover, 
one readily checks that $I$ is surjective.

\vspace{0.25cm}

\noindent{\bf I has no fixed point:} The equality $ I (x_1,x_2,\ldots) = (x_1,x_2,\ldots)$ forces $x_n = e^{\frac{2 \pi i}{n!}} (x_n - 1) + 1$, 
which easily yields $x_n=1$ for $n \geq 2$. Hence, the only candidates to fixed points are of the form $(x_1,1,1,\ldots)$, but such a 
point does not belong to $\ell^2$. 

\vspace{0.25cm}

\noindent{\bf I is recurrent:} For $r_n := n!$ we have $I^{r_n} (x_1,x_2,\ldots) = (x_1,x_2,\ldots,x_n,y_{n+1}^*,y_{n+2}^*)$, where 
$y_{n+k}^* = e^{\frac{2\pi i n!}{(n+k)!}}  (x_{n+k}  - 1) + 1$. Thus,  
$$ y_{n+k}^* - x_{n+k}  = x_{n+k} \Big( e^{\frac{2\pi i n!}{(n+k)!}} - 1 \Big)  - \Big( e^{\frac{2\pi i n!}{(n+k)!}} - 1 \Big).$$
%we have that the value of \, $\| I^{r_n}(x_1,x_2,\ldots) - (x_1,x_2,\ldots) \|$ \, is less than or equal to 
%$$4 \sum_k \| x_{n+k} \|^2 + 2 \sum_k \| x_{n+k} \| \cdot \|  \big( e^{\frac{2\pi i n!}{(n+k)!}} - 1 \big) \|^2 +  \sum_k \| \big( e^{\frac{2\pi i n!}{(n+k)!}} - 1 \big) \|^2.$$
Now, for all $n \geq 1$ and $k \geq 1$,  
$$ \left\|  e^{\frac{2\pi i n!}{(n+k)!}} - 1  \right\| 
= \left( 2 \Big( 1 - \cos \big(\frac{2\pi n!}{(n+k)!} \big) \Big) \right)^{1/2} 
\leq 2 \sin \big(\frac{2\pi n!}{(n+k)!} \big) \leq \frac{4 \pi}{(n+1)\cdots(n+k)},$$
Using this and the Cauchy-Bunyakovsky-Schwarz inequality, 
one easily concludes that the value of \, $\| I^{r_n}(x_1,x_2,\ldots) - (x_1,x_2,\ldots) \|$ \, converges to 0 as $n \to \infty$. 

\vspace{0.25cm}

Note that no power of $I$ can have a fixed point, otherwise the orbits of $I$ would be bounded and $I$ itself would have a fixed point (namely, the Chebyshev 
center of any orbit). Moreover, since $r_n$ is a multiple of $m$ for all $n \geq m$, we have that $I^m$ is also recurrent. Hence, not only $I$ but also all of its 
powers are recurrent but fixed-point-free. 

\begin{rem} \label{rem:isom-todos}
Although very explicit in this case, the last property is not particular to Edelstein's isometry. More precisely, if $I$ is recurrent, then 
all of its powers are. Indeed, if $I^{r_n}$ converges to the identity for the strong topology, then for each fixed $m \neq 0$ and every vector $v$,
\begin{small}
$$\| (I^m)^{r_n} (v) - v  \| = \Big\| \sum_{i=1}^m I^{ir_n} (v) - I^{(i-1)r_n} (v) \Big\| \leq  \sum_{i=1}^m \big\|  I^{(i-1)r_n} (I^{r_n}(v)) - I^{(i-1)r_n} (v) \big\| 
= \sum_{i=1}^{m} \big\| I^{r_n}(v)-v \big\|.$$
\end{small}Therefore, $\| (I^m)^{r_n} (v) - v \| \leq m \, \| I^{r_n} (v) - v \|$ converges to 0. 
\end{rem}

Finally, note that by considering products of Hilbert spaces and letting isometries as above acting on the factors, one can build groups 
of affine isometries isomorphic to $\mathbb{Z}^{d}$ (and, actually, to $\mathbb{Z}^{\infty}$, with some minor adjustments) all of whose 
nontrivial elements are recurrent but fixed-point-free. Alternatively, one can randomly choose some coordinates of $\ell^2$ and rotate 
them around 1 by the angles in Edlestein's example and do not rotate the other ones. Different random choices will produce 
different recurrent isometries that will actually commute. Moreover, for generic random choices, these will turn to be 
fixed-point-free (for this, it suffices to rotate infinitely many coordinates).

%%%%%%%%%%%%%%%%%%%%%%%%%%%%%%%%%%%%%%%%%%%%%%%%%%%%%%%%%%%%%%%%%%%%%%%%%%

\section{Some background on circle diffeomorphisms}
\label{s:circle}

We collect here several classical results on the dynamics of circle diffeomorphisms with irrational rotation number that will be crucial 
for the examples that we will later discuss. All maps we consider are orientation preserving. For the first result, recall that $C^{1+bv}$ 
stands for $C^1$ maps whose derivatives have bounded variation.

\begin{thm} {\bf [Denjoy]} \label{denjoy}
{\em If $f$ is a $C^{1+bv}$ circle diffeomorphism of irrational rotation number $\rho$, then the sequence of iterates $f^{q_n}$ 
converges to the identity in $C^0$ topology, where $p_n / q_n$ is the sequence of the rational approximations of $\rho$ 
obtained from its continuous fraction expansion.}
\end{thm}

Although this is not the original version of Denjoy's theorem, one readily checks that it is equivalent to it. See \cite{Na11} for a full 
discussion, including a nonstandard proof for $C^{1+Lip}$ diffeomorphisms (with no control of distortion involved) obtained in 
collaboration with C. G. Tamm de Araujo Moreira (c.f. Exercise 3.1.4 therein). 

The following  is a classical result from the seminal paper \cite{He79} of Herman. 

\begin{thm} {\bf [Herman]} 
{\em If $f$ is a $C^{2}$ circle diffeomorphism of irrational rotation number $\rho$, the the sequence of iterates $f^{q_n}$ 
converges to to the identity in $C^1$ topology, where $p_n / q_n$ is the sequence of the rational approximations of $\rho$ 
obtained from its continuous fraction expansion.}
\label{herman}
\end{thm}

This theorem extends to the $C^{1+ac}$ regularity, that is, to $C^1$ diffeomorphisms whose derivatives are absolutely continuous 
(equivalently, they have a second derivative in $L^1$). This sharpened version, which involves a simplification of Herman's proof, 
appears in \cite{Na23} (see also \cite{NT13}). We do not know whether it is still true in class $C^{1+bv}$ (we suspect it is not). 

For the results below, recall that a property suitable for points in a complete metric space is {\em generic} if it satisfied on a 
{\em $G_{\delta}$ set}, that is, on a set which is the intersection of a countable family of open and dense sets. By the Baire's 
category theorem, such a property is satisfied on a dense set of points. 

Following \cite{He79}, we denote $\mathbb{F}^{\infty}$ the complementary set in $\mathrm{Diff}^{\infty}_+ (\clo)$ of the 
interior of the set of diffeomorphisms with rational rotation number. This is a closed subset of $\mathrm{Diff}^{\infty}_+(\clo)$, hence a 
complete metric space. Also by the definition, the set of diffeomorphisms with irrational rotation number is dense in $\mathbb{F}^{\infty}$, 
but a stronger property holds: this set is actually generic. We refer to \cite[Chapter III]{He79} for some elementary properties about this space.  
Also, although not explicitly stated in \cite{He79}, the following lemma (which was privately communicated to me by Yoccoz about 20 years ago) 
is ``hidden'' in Chapter XII therein. We provide a short proof for the reader's convenience.

\begin{lem} 
{\em The set $\mathcal{R}$ of diffeomorphisms $f \in \mathbb{F}^{\infty}$ with irrational rotation number for which there is an increasing 
sequence of integers $r_n$ such that $f^{r_n}$ converges to the identity in $C^{\infty}$ topology is generic in $\mathbb{F}^{\infty}$.}
\label{lem:yoccoz}
\end{lem}

\begin{proof} Let $\mathrm{dist}_{\infty}$ be the metric defining the topology of $\mathrm{Diff}^{\infty}_+(\clo)$. For each $k \geq 1$, let 
$$\mathcal{R}_{k} := \Big\{f \in \mathbb{F}^{\infty}: \lim_{\, n \,\, \to} \! \inf_{\!\! \infty} \,\, \mathrm{dist}_{\infty} (f^n, id) < 1/k \Big\} .$$
Since $\mathcal{R} = \bigcap_k \mathcal{R}_{k}$ and $\mathcal{R}_k = \bigcap_i \bigcup_{j \geq i} \mathcal{R}_{k,j}$, where 
$$\mathcal{R}_{k,j} := \Big\{ f \in \mathbb{F}^{\infty} \!: \mathrm{dist}_{\infty} (f^{j}, id) < 1/k \Big\},$$ 
it suffices to show that each set $\bigcup_{j\geq i} \mathcal{R}_{k,j}$ is open and dense in $\mathbb{F}^{\infty}$.  Now, openess of 
$\bigcup_{j\geq i}  \mathcal{R}_{k,j}$ follows directly from its definition (actually, each $\mathcal{R}_{k,j}$ is open). To show denseness, 
it suffices to note that $\bigcup_{j\geq i}  \mathcal{R}_{k,j}$ contains all diffeomorphisms with an irrational rotation number satisfying 
a Diophantine condition (which is a dense subset of $\mathbb{F}^{\infty}$). Indeed, according to the celebrated theorem of Herman 
\cite{He79} and Yoccoz \cite{Yo84}, such a diffeomorphism is conjugate to the corresponding rotation by a $C^{\infty}$ 
diffeomorphism. Now, rotations are obviously contained in $\bigcup_{j\geq i} \mathcal{R}_{k,j}$, and one readily 
checks that diffeomorphisms that are conjugate to rotations also lie in $\bigcup_{j\geq i} \mathcal{R}_{k,j}$. 
\end{proof}

In full generality, diffeomorphisms for which there is a sequence of iterates converging to the identity are called {\em rigid} (they are also called 
{\em recurrent} in specific contexts). As shown above, for the case of circle maps, these form a generic subset of $\mathbb{F}^{\infty}$. 

\begin{rem}\label{rem:all-powers} Note that if $f$ is rigid than all powers $f^m$ are also rigid. Indeed, since $\mathrm{Diff}^{\infty}_+ (\clo)$ is 
a topological group, if $f^{r_n}$ converges to the identity then so is the case of $(f^m)^{r_n} = (f^{r_n})^m$ for each nonzero 
$m \in \mathbb{Z}$. (Compare Remark \ref{rem:isom-todos}.)
\end{rem}

Despite the above, generic diffeomorphisms in $\mathbb{F}^{\infty}$ also share very bad properties along certain iterates, which reflect, 
for example, in the following important result also from \cite[Chapter XII]{He79}.

\begin{thm} {\bf [Herman]} 
{\em A generic element in $\mathbb{F}^{\infty}$ is conjugated to a rotation only by homeomorphisms that are not absolutely continuous.}
\label{thm:not-smooth}
\end{thm}

Here is another manifestation of a similar phenomenon that will be useful for us. For the statement, recall that the {\em crossratio}  
of four cyclically ordered points $a < b < c < d < a$ on $\clo$ is defined as 
$$[a,b,c,d] : = \frac { (a-c) \, (b-d)}{(a-d) \, (b-c)}.$$
 
\begin{lem} \label{lem:quasim}
{\em The set $\mathcal{D}$ made by the diffeomorphisms $f \in \mathbb{F}^{\infty}$ of irrational rotation number such that
$$\lim_{\, \, n \,\, \to} \! \sup_{ \!\! {}^\infty} \, \max_{[a,b,c,d] = 2} \, \, [f^n(a),f^n(b),f^n(c),f^n(d)] = \infty$$
is generic in $\mathbb{F}^{\infty}$.} 
\end{lem}

\begin{proof} For each $k \geq 1$, let $\mathcal{D}_k$ denote the set of $f \in \mathbb{D}^{\infty}$ 
$$\lim_{\,\, n \,\,\, \to} \! \sup_{ \!{}^\infty} \, \max_{[a,b,c,d] = 2} \, \, [f^n(a),f^n(b),f^n(c),f^n(d)] > k. $$
Since $\mathcal{D} = \bigcap_k \mathcal{D}_k$ and $\mathcal{D}_k = \bigcap_i \bigcup_{j \geq i} \mathcal{D}_{k,j}$, 
where 
$$\mathcal{D}_{k,j} := \Big\{ f \in \mathbb{F}^{\infty} \!: \max_{[a,b,c,d]=2} \, [f^j(a),f^j(b),f^j(c),f^j(d)] > k \Big\},$$
it suffices to show that each set $\bigcup_{j\geq i}  \mathcal{D}_{k,j}$ is open and dense in $\mathbb{F}^{\infty}$. Now, openess of 
each $\mathcal{D}_{k,j}$ readily follows from the definition. To show denseness of  $\bigcup_{j\geq i}  \mathcal{D}_{k,j}$, 
it suffices to check that nonperiodic diffeomorphisms of rational rotation number with denominator at least 3 which are not 
conjugate to a rotation are contained therein, because these maps are known to form a dense subset of $\mathbb{F}^{\infty}$ 
(see \cite[Chapter 3]{He79} as well as \cite{Ar61}). Let hence $f$ be such a diffeomorphism, and let $q \geq 3$ be the period 
of its periodic points. Let $(u,v)$ be a connected component of the complement of the set of fixed points of $f^q$. Then $u,v$ 
are fixed under $f^q$, and since $q \geq 3$, there exists another fixed point $w$ of $f^q$ in the interval $(v,u)$. Two cases 
are possible:
\begin{itemize}
\item If $v$ is a topologically contracting from the right for $f^q$, let $a:=v$, $b:=w$ and $c:=u$.  
Then define $d$ as being the unique point in $(c,a)$ such that $[a,b,c,d] = 2$. 
\item If $v$ is a topologically repelling from the right for $f^q$, let $a:=w$, $b:=u$ and $d:=v$.  
Then define $c$ as being the unique point in $(b,d)$ such that $[a,b,c,d] = 2$. 
\end{itemize}
In the former case, since $a,b,c$ are fixed points of $f^q$, while $f^{nq} (d) \to a$ as $n \to \infty$, we have 
$$[ f^{nq}(a), f^{nq}(b), f^{nq}(c), f^{nq}(d)]  \to \infty \quad \mbox{ as } \quad n \to \infty,$$
thus showing that $f$ lies in $\bigcup_{j\geq i}  \mathcal{D}_{k,j}$. The later case can be treated similarly. 
\end{proof}

%%%%%%%%%%%%%%%%%%%%%%%%%%%%%%%%%%%%%%%%%%%%%%%%%%%%%%%%%%%%%%%%%%%%%%%%%%%%%%%%

\section{Examples of fixed-point-free recurrent isometries}
\label{s:banachs}

\subsection{An example in $C^0 (\clo)$}
\label{s:C0}

Crucially, we will use below the chain rule (cocycle identity) for this {\em logarithmic derivative} of $C^1$ maps: 
\begin{equation}\label{eq:log}
\log D (g_1 \circ g_2) = \log Dg_2 + (\log D g_1) \circ g_2.
\end{equation}

Fix any $C^{1+ac}$ circle diffeomorphism $f$ with irrational rotation number $\rho$ that is not $C^1$ conjugate to the corresponding rotation (the 
existence of which is vastly ensured by Theorem~\ref{thm:not-smooth} above). Consider the Koopman operator $\Theta = \Theta_f$ on $C^0(\clo)$ 
given by $\varphi \mapsto \varphi \circ f$ (this is a linear isometric opeator). Let $c \in C^0 (\clo)$ be the continuous function $x \mapsto \log Df (x)$. 
Let $I  := \Theta + c$ be the associated affine isometry of $C^0(\clo)$. We claim that $I$ has no fixed point, yet one has the convergence of 
$I^{q_n}$ to the identity as $n \to \infty$ (where $p_n/q_n$ is the rational approximation of $\rho$ coming from its continuous fraction expansion). 

To see that $I^{q_n}$ converges to the identity, note that the chain rule again easily yields 
$$I^n (\varphi) = \varphi \circ f^n + \sum_{i=0}^{n-1} \log Df \circ f^i = \varphi \circ f^n + \log Df^n = \Theta^n (\varphi) + c_n.$$ 
By Theorem \ref{denjoy}, $f^{q_n}$ uniformly converges to the identity, hence $\Theta^{q_n} (\varphi) = \varphi \circ f^{q_n}$ converges 
to $\varphi$ for each fixed $\varphi \in C^0 (\clo)$. By Theorem \ref{herman}, $c_{q_n} = \log Df^{q_n}$ uniformly converges to zero.

Next, asume that $I$ has a fixed point $\varphi \in C^0 (\clo)$. Then $I (\varphi) = \varphi$ translates into 
\begin{equation}\label{eq:fi}
\varphi \circ f = \varphi + \log Df.
\end{equation}
Consider the $C^1$ circle diffeomorphism $h$ defined as 
$$h (x) := \frac{\int_0^x e^{- \varphi (s)} ds}{\int_0^1 e^{-\varphi(s)}ds} .$$
Then for a certain constant $\kappa$ we have \, $\log Dh = -\varphi  + \kappa$. \,Thus, (\ref{eq:fi}) translates into 
$$- \log Dh \circ f = - \log Dh + \log Df,$$
which, by the chain rule, can be rewriten as 
$$0 = \log Dh \circ f + \log Df - \log Dh = \log D (hfh^{-1}) \circ h^{-1}.$$
Thus, $D (hfh^{-1})$ identically equals 1, which means that it is a rotation. 
However, this contradicts the fact that $f$ is not $C^1$ conjugate to a rotation.

%%%%%%%%%%%%%%%%%%%%%%%%%%%%%%%%%%%%%%%%%%%%%%%%%%%%%%%%%%%%%%%%%%%%%%%%%%%%%%%

\subsection{An example in $L^1 (\clo)$}
\label{s:L1}

We will now use the {\em affine derivative} $D \log Df = D^2 f / Df$ and its associated chain rule 
\begin{equation}\label{eq:affine}
\frac{D^2 (g_1 \circ g_2)}{D (g_1 \circ g_2)} = \frac{D^2 g_2}{D g_2} + \left( \frac{D^2 g_1}{D g_1} \right) \circ g_2 \cdot Dg_2.
\end{equation}
Note that this cocycle identity follows from (\ref{eq:log}) just by taking derivatives.

Now fix any $C^2$ circle diffeomorphism $f$ with irrational rotation number that is not Lipschitz conjugate to the corresponding rotation and for 
which there exists an increasing sequence of integers $r_n$ such that $f^{r_n}$ converges to the identity in $C^2$ topology (the existence 
of such an $f$ is ensured by Lemma \ref{lem:yoccoz}  and Theorem \ref{thm:not-smooth} above). Consider the linear isometric operator 
(regular representation) $\Theta = \Theta_f$ on $L^1 (\clo)$ given 
by $\psi \mapsto \psi \circ f \cdot Df$. Let $c \in L^1 (\clo)$ be the continuous (hence integrable) function $x \mapsto \log D^2 f  (x) / D f(x)$, 
and let $I  := \Theta + c$ be the associated affine isometry of $L^1 (\clo)$. Again, we claim that $I$ has no fixed point, yet one has the convergence 
of $I^{r_n}$ to the identity as $n \to \infty$.

To see that $I^{r_n}$ converges to the identity, note that the chain rule (\ref{eq:affine}) now yields 
$$I^n (\psi) 
= \psi \circ f^n \cdot Df^n + \sum_{i=0}^{n-1} \left( \frac{D^2 f }{Df} \right) \circ f^i \cdot Df^i 
= \psi \circ f^n \cdot Df^n + \frac{D^2 f^n}{D f^n} = \Theta^n (\psi) + c_n.$$ 
On the one hand, since $f^{r_n}$ converges to the identity in $C^1$ topology, $\Theta^{r_n} (\psi) = \psi \circ f^{r_n} \cdot Df^{r_n}$ 
converges to $\psi$ for each fixed $\psi \in L^1 (\clo)$. Indeed, this is easy to check for a continuous $\psi$, and readily extends to 
$L^1$ since $\Theta$ is isometric. On the other hand, since $f^{r_n}$ converges to the identity in $C^2$ topology, 
the sequence $c_{r_n} = D^2f^{r_n} / Df^{r_n}$ converges uniformly (hence in $L^1$) to zero.

Next, asume that $I$ has a fixed point $\psi \in L^1(\clo)$. Then $I (\psi) = \psi$ translates into 
\begin{equation}\label{eq:fi2}
\psi \circ f \cdot Df = \psi + \frac{D^2 f}{D f}.
\end{equation}
Consider the circle map $h$ defined as 
$$h (x) := \frac{\int_0^x e^{- \int_0^s \psi (t) dt} ds}{\int_0^1 e^{-\int_0^s \psi(t) dt}ds} .$$
It is easy to see that $h$ is a $C^1$ diffeomorphism except perhaps at the point $0 \sim 1$, where lateral derivatives exist 
(and are nonzero) but may be different (see Remark \ref{rem:no-disc} on this).  Actually, on $(0,1)$, the map $h$ is $C^{1+ac}$, 
with $D^2 h / Dh = - \psi$. 
Equation (\ref{eq:fi2}) then translates into 
$$- \frac{D^2 h}{Dh} \circ f \cdot Df = - \frac{ D^2h}{Dh} + \frac{D^2 f}{Df},$$
which, by the chain rule (\ref{eq:affine}), can be rewriten as 
$$0 
= \frac{D^2 h}{Dh} \circ f \cdot Df + \frac{D^2 f}{D f} - \frac{D^2 h}{Dh} 
= \frac{D^2 (hfh^{-1}) }{D (hfh^{-1}) } \circ h^{-1} \cdot Dh^{-1} .$$
Thus, $D^2 (hfh^{-1})$ equals 0 a.e., which means that it is a rotation. 
However, this contradicts the fact that $f$ is not Lipschitz conjugate to a rotation.

\begin{rem} \label{rem:no-disc}
A priori, the conjugacy $h$ above may fail to be $C^1$ at $0 \sim 1$ because of the nonvanishing of the integral of $\psi$. 
However, this cannot occur, because the set of continuity points of $Dh$ is invariant under $f$. Another way to see this is by 
using a classical argument from \cite{He79} which implies that every Lipschitz conjucacy of a $C^2$ circle diffeomorphism of 
irrational rotation number to the corresponding rotation is actually smooth (see \cite{Na07} for further developments on this). 
\end{rem}

%%%%%%%%%%%%%%%%%%%%%%%%%%%%%%%%%%%%%%%%%%%%%%%%%%%%%%%%%%%%%%%%%%%%%%%%%%%%%%%

\subsection{An example in $L^2 (\clo \times \clo)$}
\label{s:L2}

We will now use a kind of projective derivative associated to a circle diffeomorphism $f$. 
Concretely, this is the function $c_f$ defined a.e. on $\clo \times \clo$ as 
$$c_f (x,y) = \frac{\sqrt{ Df(x) Df (y) } }{\mathrm{dist} (f(x),f(y))} - \frac{1}{\mathrm{dist} (x,y)}.$$
The cocycle relation is 
\begin{equation}\label{eq:proy}
c_{g_1 \circ g_2} (x,y) = c_{g_2} (x,y) + c_{g_1} (g_2 (x), g_2 (y)) \cdot \sqrt{Dg_2 (x) \, Dg_2 (y)}.
\end{equation} 
One can show that $c_f$ belongs to $L^2 (\clo \times \clo)$ for all $C^2$ diffeomorphism $f$. 
Moreover, an elementary estimate shows that there exists a constant $C$ such that 
\begin{equation}\label{eq:estimate}
\| c_f \|_{L^2} \leq C \| D^2 f \|_{C^0} 
\,\, \mbox{ holds provided } \,\, \|Df - 1\|_{C^0} \leq 1/2
\end{equation}
(see for instance \cite[Proposition 2.1]{Na02}). 

Now fix any $C^2$ circle diffeomorphism $f$ with irrational rotation number such that
\begin{equation}\label{eq:quasym-dist}
\lim_{\, \, n \,\, \to} \! \sup_{ \!\! {}^\infty} \, \max_{[a,b,c,d] = 2} \, \, [f^n(a),f^n(b),f^n(c),f^n(d)] = \infty
\end{equation}
(c.f. Lemma \ref{lem:quasim})
and for which there exists an increasing sequence of integers $r_n$ such that $f^{r_n}$ converges to the identity in $C^2$ topology (c.f. 
Lemma \ref{lem:yoccoz}). Consider the linear isometry (regular representation) $\Theta = \Theta_f$ of $L^2 (\clo \times \clo)$ 
given by 
$$\Theta (\chi) (x,y) := \chi (f(x), f(y)) \cdot \sqrt{Df(x) Df(y)}.$$ 
For $c=c_f$, consider the affine isometry of $L^2 (\clo \times \clo)$ given by $I:= \Theta +c$. 
Again, we claim that $I$ has no fixed point, yet one has the convergence of $I^{r_n}$ to the identity as $n \to \infty$.

To see that $I^{q_n}$ converges to the identity, note that the cocycle relation (\ref{eq:proy}) now yields 
\begin{eqnarray*}
I^n (\chi) (x,y) 
&=& \chi (f^n (x), f^n (y)) \cdot \sqrt{Df^n (x) Df^n (y)} + \sum_{i=0}^{n-1} c ( f^i (x) , f^i (y) ) \sqrt{Df^i (x) Df^i (y)} \\
&=&   \chi (f^n (x), f^n (y)) \cdot \sqrt{Df^n (x) Df^n (y)} + c_{f^n} (x,y), 
\end{eqnarray*}
hence
$$I^n (\chi) = \Theta^n (\chi) + c_{f^n}.$$
On the one hand, since $f^{r_n}$ converges to the identity in $C^1$ topology, we have that $\Theta^{r_n} (\chi)$ converges to 
$\chi$ for each fixed $\chi \in L^2 (\clo \times \clo)$ (again, this is easy to check for continuous $\chi$, and readily extends to the 
whole space $L^2$). On the other hand, since $f^{r_n}$ converges to the identity in $C^2$ topology, the inequality (\ref{eq:estimate}) 
holds for $f^{r_n}$ provided $n$ is large enough. Obviously, this shows that $\| c_{f^{r_n}} \|_{L^2} \leq C \| D^2 f^{r_n} \|_{C^0}$ 
converges to 0.

Next, asume that $I$ has a fixed point $\chi \in L^2 (\clo \times \clo)$. Then $I (\chi) = \chi$ translates into 
$$\chi (f(x), f(y)) \cdot \sqrt{Df(x) Df(y)} = \chi (x,y) + c_f (x,y),$$ 
that is 
\begin{equation}\label{eq:fi3}
\left[ \frac{1}{\mathrm{dist} (f(x),f(y))} - \chi (f(x), f(y)) \right]^2  \cdot Df(x) Df(y) = \left[ \frac{1}{\mathrm{dist}(x,y)} - \chi (x,y) \right]^2 
\end{equation}
for a.e. $(x,y) \in \clo \times \clo$. The contradiction is provided by the next lemma from \cite{Na05} 
(for the reader's convenience, we give a short proof of it).

\begin{lem} \label{lem:incompatibles}
{\em Conditions (\ref{eq:quasym-dist}) and (\ref{eq:fi3}) are incompatible for a circle diffeomorphism 
$f$ (and a function $\chi \in L^2 (\clo \times \clo)$).} 
\end{lem}

\begin{proof} We will strongly use the following elementary yet crucial formula valid for all $a<b<c<d<a$ on $\clo$ 
(see \cite{Bo88} for details): 
\begin{equation}  \log ([a,b,c,d]) 
=  \int_a^b \int_c^d \frac{dx \, dy}{\mathrm{dist}^2 (x,y)} .
\label{eq:fi4} 
\end{equation}

By the cocycle property, equality (\ref{eq:fi3}) implies that 
$$
\left[ \frac{1}{\mathrm{dist} (f^n(x),f^n(y))} - \chi (f^n(x), f^n(y)) \right]^2  \cdot Df^n(x) Df^n(y) = \left[
 \frac{1}{\mathrm{dist}(x,y)} - \chi (x,y) \right]^2 
$$
holds for all $n \geq 1$ (and a.e. $(x,y)$). If $[a,b,c,d]= 2$, then this combined with (\ref{eq:fi4}) and 
the triangle inequality (applied twice) yield that 
$$
 \log ([f^n(a),f^n(b),f^n(c),f^n(d)])^{1/2} 
= \left[ \int_{a}^{b} \int_{c}^{d} \frac{Df^n (x) Df^n (y)}{\mathrm{dist}^2(f^n(x),f^n(y))} \right]^{1/2}
$$
is smaller than or equal to 
\begin{small}
\begin{eqnarray*}
&& \hspace{-1.9cm} \left[ \int_{a}^{b}\int_c^d  \left[ \frac{1}{\mathrm{dist} (f^n(x),f^n(y))} - \chi (f^n(x), f^n(y)) \right]^2 Df^n(x) Df^n(y) \right]^{1/2} \\
         && \hspace{3.5cm}+  \left[ \int_{a}^{b} \int_c^d \chi^2 (f^n(x), f^n(y)) Df^n(x) Df^n(y) \right]^{1/2} \\
&=& \left[ \int_{a}^{b}\int_c^d \left[ \frac{1}{\mathrm{dist} (x,y)} - \chi (x, y) \right]^2 \right]^{1/2}
            +  \left[ \int_{f^n(a)}^{f^n(b)} \int_{f^n(c)}^{f^n(d)} \chi^2 (x,y) \right]^{1/2} \\
&\leq &
\left[ \int_{a}^{b}\int_c^d \left[ \frac{1}{\mathrm{dist} (x,y)} - \chi (x, y) \right]^2 \right]^{1/2}   +  \| \chi \|_{L^2} \\
&\leq& 
\left[ \int_{a}^{b}\int_c^d  \frac{1}{\mathrm{dist}^2 (x,y)}  \right]^{1/2} 
     + \left[ \int_{a}^{b}\int_c^d  \chi^2 (x,y)  \right]^{1/2} 
            +  \| \chi \|_{L^2} \\
&\leq& 
 \big| \log ([a,b,c,d] )\big|^{1/2}  + 2 \| \chi \|_{L^2}  \\
&=&  | \log (2) |^{1/2} +  2 \| \chi \|_{L^2} .
\end{eqnarray*}
\end{small}Therefore, 
\begin{equation}\label{eq:arriba}
[f^n(a),f^n(b),f^n(c),f^n(d)]
\leq 
e^{\big( | \log (2) |^{1/2} +  2 \| \chi \|_{L^2} \big)^2},
\end{equation}
which is in contradiction to (\ref{eq:quasym-dist}). 
\end{proof}

\vspace{0.1cm}

\begin{rem} \label{rem:incompatibles}
A minoration argument in  the triangle inequality above shows that (\ref{eq:fi3}) not only implies (\ref{eq:arriba}) but also 
\begin{equation}\label{eq:abajo}
[f^n(a),f^n(b),f^n(c),f^n(d)]
\geq 
e^{\big( | \log (2) |^{1/2} -  2 \| \chi \|_{L^2} \big)^2.}.
\end{equation}
\end{rem}

\begin{rem} The construction above is maleable for other $L^p$ spaces. More precisely, 
consider the regular representation 
$$\Theta_p (\chi) (x,y) := \chi (f(x), f(y)) \cdot \sqrt[^{1/p}]{Df(x) Df(y)}$$ 
and the cocycle 
$$c_p (x,y) := \frac{\sqrt[^{1/p}]{ Df(x) Df (y) } }{\mathrm{dist}^{1/p} (f(x),f(y))} - \frac{1}{\mathrm{dist}^{1/p} (x,y)}.$$
Now consider the affine isometry $I_p := \Theta_p + c_p$ of $L^{p} (\clo \times \clo)$. All arguments above work for $1 < p < \infty$ 
provided $f$ is of class $C^2$. The case $p=1$ is quite interesting. One can readily show that $c_1$ lies in $L^1 (\clo \times \clo)$ 
provided $f$ is in $\mathrm{Diff}^{2+\alpha}_+(\clo)$ for some $\alpha > 0$, but this is no longer true for general $C^2$ diffeomorphisms. 
(Producing an explicit example of failure of integrability of $c_1$ for a $C^2$ diffeomorphism is a good exercise.)
This becomes particularly transparent for $C^3$ diffeomorphisms. Indeed, for such maps, $c_1$ allows retrieving the 
Schwarzian derivative $Sf$: 
$$\lim_{y \to x} c_{1} (x,y) 
= \lim_{y \to x} \left[ \frac{\sqrt{ Df(x) Df (y) } }{\mathrm{dist} (f(x),f(y))} - \frac{1}{\mathrm{dist} (x,y)} \right] 
= \frac{Sf (x)}{6} = \frac{1}{6} \left[ \frac{D^3 f}{D f} - \frac{3}{2} \left( \frac{D^2 f}{D f} \right)^2 \right] .$$
Moreover, the cocycle relation (\ref{eq:proy}) translates into the classical chain rule 
$$S (g_1 \circ g_2) = S (g_2) + S (g_1) \circ g_2 \cdot (Dg_2)^2.$$
See \cite{Na06} for more on this. 
\end{rem}

%%%%%%%%%%%%%%%%%%%%%%%%%%%%%%%%%%%%%%%%%%%%%%%%%%%%%%%%%%%%%%%%%%%%%%%%%%%%%%%%%%%

\section{Adding commutativity}
\label{s:com} 

The affine isometries we have just constructed are fixed-point free yet recurrent for a generic subset of $\mathbb{F}^{\infty}$. 
Due to the cocycle property of the logarithmic, affine and projective derivative, in order to build commuting such isometries,  
we need to ensure that some of these diffeomorphisms have a large centralizer and many elements therein satisfy analogous 
properties, as for instance rigidity. We explain below the necessary adjustments (with an increasing level of difficulty) for each case. 
Below, we will use without much details the fact that the centralizer of every $C^{1+bv}$ circle diffeomorphism of irrational rotation 
number is topologically conjugate to the group of rotations, hence Abelian (this is a consequence of Denjoy's theorem together with 
the fact that the centralizer of an irrational rotation is the group of rotations). Moreover, this conjugacy is unique up to composition 
with a rotation, hence its degree of regularity is well defined.

%%%%%%%%%%%%%%%%%%%%%%%
\subsection{\bf The $C^0$ case.}

We fix a diffeomorphism $f \in \mathbb{F}^{\infty}$ of irrational rotation number that is rigid (c.f. Lemma \ref{lem:yoccoz}) and for which 
the conjugacy to the corresponding rotation is not $C^1$ (c.f. Lemma \ref{thm:not-smooth}). Let us consider the $C^{\infty}$ closure 
$\overline{ \langle f \rangle }$ of the group generated by $f$. This is an  Abelian topological group contained in the $C^{\infty}$ centralizer 
of $f$. Moreover, it has no isolated point, hence its cardinality is that of the continuum. Let $f_1,f_2,\ldots$ be a sequence of independent 
elements therein so that they generate a group isomorphic to $\mathbb{Z}^{\infty}$. For each $i \geq 1$, let $I_i := \Theta_i + c_i$ be the 
isometry of $C^0 (\clo)$ defined by $I_i (\varphi) = \varphi \circ f_i + \log Df_i$. Property (\ref{eq:log}) gives the commutativity relation 
$I_i I_j = I_j I_i$ for all $i,j$, and looking only at the linear part it is easy to see that $I_1,I_2,\ldots$ freely generate a group $\mathcal{I}$ 
of isometries of $C^0(\clo)$ isomorphic to $\mathbb{Z}^{\infty}$. Every element of $\mathcal{I}$ is of the form $I_{g}$ for a certain 
$g \in \langle f_1,f_2,\ldots \rangle$, where $I_g (\varphi) = \varphi \circ g + \log Dg$. Using Theorem \ref{herman} one then concludes that 
$I(g^{q_n})$ converges to the identity for every nontrivial $g$, where $q_n$ is the sequence of denominators of the continuous fraction 
approximation of the rotation number of $g$ (which is necessarily irrational). Finally, none of the $I(g)$ (with $g$ nontrivial) has a fixed 
point, otherwise, arguing as in \S \ref{s:C0}, such a $g$ would be $C^1$ conjugate to the rotation, hence $f$ would also be (by the same 
conjugating map). 

\begin{rem} A deep theorem of Yoccoz \cite{Yo84} establishes that, for a generic $f \in \mathbb{F}^{\infty}$, the closure 
$\overline{\langle f \rangle}$ coincides with the $C^{\infty}$ centralizer of $f$. Although this is not used neither in the 
proof above nor in those of the subsequences sections, it helps to clarify the context we are dealing with.
\end{rem}

%%%%%%%%%%%%%%%%%%%%%%%
\subsection{\bf The $L^1$ case.}

The construction above just requires $C^1$ convergence to the identity of a sequence of iterates, which is always ensured by 
Theorem \ref{herman}. Here, however, we need $C^2$ convergence, which only holds generically. Thus, if we start with an 
element $f$ with a large centralizer, it may happen that some of the elements of $\overline{\langle f \rangle}$ fail to be rigid. 
Let us state this as an explicit question of a certain independent interest (we suspect an affirmative answer, yet it seems hard 
to produce explicit examples). 

\begin{qs} {\em Does there exist a rigid diffeomorphism $f \in \mathbb{F}^{\infty}$ for which there is 
a nonrigid $g \in \overline{ \langle f \rangle }$~?}
\end{qs}

The idea to bypass this problem is, essentially, to show that the hypothetical situation thus described does not 
arise generically in an appropriate context. 

\begin{prop} \label{lem:mas}
{\em There exists a sequence $f_1,f_2,\ldots$ of commuting elements in $\mathbb{F}^{\infty}$ with irrational rotation number 
that freely generate a group isomorphic to $\mathbb{Z}^{\infty}$ such that each $g \in \langle f_1,f_2,\ldots \rangle$ is rigid and non $C^1$ 
conjugate to a rotation if nontrivial.}
\end{prop}

\begin{proof} We let $f_1$ be an element of $\mathbb{F}^{\infty}$ for which the conjugacy to the rotation is not $C^1$ (c.f. Lemma 
\ref{thm:not-smooth}) and which is rigid (c.f. Lemma \ref{lem:yoccoz}) as well as all of its powers (c.f. Remark \ref{rem:all-powers}). 
Below we show by induction the existence of $f_2,f_3,\ldots$ such that $f_1,\ldots,f_d$ freely generate a copy of $\mathbb{Z}^d$ 
and each $g \in \langle f_1, \ldots , f_d \rangle$ is rigid. The fact that such a nontrivial  $g$ is not $C^1$-conjugate to a rotation 
will then follow from the uniqueness of the conjugacy (up to rotations), as explained above.

The inductive step is carried out by showing that, starting with $f_1,\ldots,f_d$, a generic element in $\overline{\langle f_1, \ldots, f_d \rangle}$ 
satisfies the desired properties. To see this, for each $k \geq 1$ and all $m_1, \ldots, m_d, m$ in $\mathbb{Z}$, consider the set 
$$\mathcal{R}^{m_1,\ldots,m_d,m}_k = \bigcap_{i \geq 1} \bigcup_{j \geq i}  \Big\{ g \in \overline{\langle f_1, \ldots , f_d \rangle} \! : 
\, \mathrm{dist}_{\infty} ( (f_1^{m_1} \cdots f_d^{m_d} g^m )^j , id ) < 1 / k \Big\}.$$
Since each set \, 
$ \big\{ g \in \overline{\langle f_1, \ldots , f_d \rangle} \! : 
\mathrm{dist}_{\infty} ( (f_1^{m_1} \cdots f_d^{m_d} g^m )^j , id ) < 1 / k \big\}$ \, 
is open, the intersection 
\begin{equation}\label{eq:inter}
\bigcap_{m_1} \ldots \bigcap_{m_d} \bigcap_{m} \bigcap_{k \geq 1} \mathcal{R}^{m_1,\ldots,m_d,m}_k
\end{equation}
is a $G_{\delta}$-set. Obviously, this intersection contains $\langle f_1, \ldots, f_d \rangle$. Therefore, it is dense 
in $\overline{\langle f_1, \ldots, f_d \rangle}$, hence generic in there. Since this group has the cardinality of the 
continuum (because it has no isolated point) and is Abelian (because it is contained in the centralizer of $f_1$), 
one can choose $f_{d+1}$ therein so that $f_1,\ldots,f_{d+1}$ freely generate a copy of $\mathbb{Z}^{d+1}$ (in particular, the rotation 
number of $f_{d+1}$ is irrational). Finally, the fact that $f_{d+1}$ belongs to the intersection (\ref{eq:inter}) means that each 
$g \in \langle f_1 \ldots, f_{d+1} \rangle$ is rigid, thus closing the proof.
\end{proof}

As above, we can now proceed to build, for each $g \in \langle f_1, f_2, \ldots \rangle$, an affine  isometry $I_g$ of $L^1 (\clo)$ by letting
$$I_g (\psi) = (\psi \circ g ) \cdot Dg + \frac{D^2 g}{D g}.$$
The cocycle relation (\ref{eq:affine}) implies that this is actually a (faithful) group action of $\mathbb{Z}^{\infty}$.  The fact that $I_g$ is recurrent 
but has no fixed point for every nontrivial $g \in \langle f_1, f_2, \ldots \rangle$ proceeds as in \S \ref{s:L1}. We leave the details to the reader.

%%%%%%%%%%%%%%%%%%%%%%%
\subsection{The $L^2$ case.}

We first need a preparation lemma, the proof of which consists of a straightforward extension 
of the argument of Lemma~\ref{lem:quasim}.

\begin{lem} 
{\em The set $\mathcal{D}^*$ made by the diffeomorphisms $f \in \mathbb{F}^{\infty}$ of irrational rotation number such that, 
for all $m \neq 0$ ($m \in \mathbb{Z}$), one has 
$$\lim_{\, \, n \,\, \to} \! \sup_{ \!\! {}^\infty} \, \max_{[a,b,c,d] = 2} \, \, [f^{m \,n}(a),f^{m \,n}(b),f^{m\, n}(c),f^{m\, n}(d)] = \infty$$
is generic in $\mathbb{F}^{\infty}$.} 
\label{lem:quasim}
\end{lem}

\begin{proof}
For each $k \geq 1$ and $m \neq 0$, we now denote $\mathcal{D}_k^m$ the set of $f \in \mathbb{D}^{\infty}$ such that 
$$\lim_{\,\, n \,\,\, \to} \! \sup_{ \!{}^\infty} \, \max_{[a,b,c,d] = 2} \, \, [f^{m\,n}(a),f^{m\,n}(b),f^{m\,n}(c),f^{m\,n}(d)] > k. $$
Since $\mathcal{D}^* = \bigcap_{m \neq 0} \bigcap_k \mathcal{D}_k^m$ and $\mathcal{D}_k^m = \bigcap_i \bigcup_{j \geq i} \mathcal{D}_{k,j}^m$, 
where 
$$\mathcal{D}_{k,j}^m := \Big\{ f \in \mathbb{F}^{\infty} \!: \max_{[a,b,c,d]=2} \, [f^{mj}(a),f^{mj}(b),f^{mj}(c),f^{mj}(d)] > k \Big\},$$
it suffices to show that each set $\bigcup_{j\geq i}  \mathcal{D}_{k,j}^m$ is open and dense in $\mathbb{F}^{\infty}$. 
This is proved by an argument analogous to that of Lemma \ref{lem:quasim}. We leave the details to the reader.
\end{proof}

Starting with an element $f_1 := f$ provided by this lemma which is also rigid (c.f. Lemma~\ref{lem:yoccoz}), let us perform the construction 
of the proof of Proposition \ref{lem:mas}. We thus obtain $C^{\infty}$ diffeomorphisms $f_1,f_2,\ldots$ that freely generate a group isomorphic 
to $\mathbb{Z}^{\infty}$ in which every element is rigid. We claim that, for each nontrivial $g \in \langle f_1, f_2, \ldots \rangle$, at least one 
of the following conditions hold: 
\begin{small}
\begin{equation}\label{eq:una-de-dos}
\lim_{\, \, n \,\, \to} \! \inf_{ \!\! \infty} \, \min_{[a,b,c,d] = 2} \, \, [g^n(a),g^n(b),g^n(c),g^n(d)] = 0, \quad 
\lim_{\, \, n \,\, \to} \! \sup_{ \!\! {}^\infty} \, \max_{[a,b,c,d] = 2} \, \, [g^n(a),g^n(b),g^n(c),g^n(d)] = \infty.
\end{equation}
\end{small}Otherwise, the group $\langle g \rangle$ would be uniformly quasisymmetric in the terminology of \cite{Hi90}. 
A deep theorem therein combined with some reformulations on the definition of quasisymmetry (see for instance \cite{LV73} 
as well as \cite{Za88}) would then imply that $g$ is conjugate to the rotation by a homeomorphism $h$ satisfying
$$C([a,b,c,d]) \leq [h(a),h(b),h(c),h(d)] \leq D ([a,b,c,d])$$
for certain  continuous increasing functions $C,D$ from $(0,\infty)$ to $(0,\infty)$ (where $a<b<c<d<a$ are arbitrary). 
However, since $\langle f_1, f_2, \ldots \rangle$ is contained in the centralizer of $f$, the same $h$ would conjugate $f$ 
to the corresponding rotation $R$. This would imply that, for all $a<b<c<d<a$ satisfying $[a,b,c,d]=2$, we would have  
\begin{small}
$$[f^n (a) , f^n (b), f^n (c), f^n (d)] = [ h^{-1}R^nh(a), h^{-1}R^nh(b), h^{-1}R^nh(c), h^{-1}R^nh(d)] \leq C^{-1}(D([a,b,c,d])),$$
\end{small}which contradicts our choice of $f$. 

Finally, let us consider the affine isometric (faithful) action of $\langle f_1, f_2, \ldots \rangle$ on $L^2 (\clo \times \clo)$ given by 
$$I_g (\chi) (x,y) := \chi (g(x), g(y)) \cdot \sqrt{Dg(x) Dg(y)} + \frac{\sqrt{ Dg(x) Dg (y) } }{\mathrm{dist} (g(x),g(y))} - \frac{1}{\mathrm{dist} (x,y)}.$$ 
As before, rigidity implies that each $I_g$ is a recurrent isometry. We claim, however, that $I_g$ has no fixed point for each nontrivial $g$. 
Indeed, otherwise, arguing as in the proof of Lemma \ref{lem:incompatibles} and Remark \ref{rem:incompatibles}, we would get an 
uniform upper bound and a uniform away-from-zero lower bound for $[g^n(a),g^n(b),g^n(c),g^n(d)]$ for all points $a<b<c<d<a$ 
satisfying $[a,b,c,d] = 2$ (independently of $n$). However, this property (that would reflect in inequalities of the form 
(\ref{eq:arriba}) and (\ref{eq:abajo}) for $g$) would violate (\ref{eq:una-de-dos}).

%%%%%%%%%%%%%%%%%%%%%%%

\vspace{0.4cm}

\noindent{\bf Acknowledgments.} 
This work was funded by the ECOS Research Project 23003 ``Small spaces under action''. I also acknowledge 
the University of Geneva and the Institut Fourier (Grenoble) for the hospitality during the preparation of this text. 

%%%%%%%%%%%%%%%%%%%%%%%%%%%%%%%%%%%%%%%%%%%%%%%%%%%%%%%%%%%%%%%%%%%%%%%%

\begin{footnotesize}

\vspace{0.1cm}

\noindent Andr\'es Navas\\ 

\noindent Universidad de Santiago de Chile\\ 

\noindent Alameda 3363, Santiago, Chile\\ 

\noindent email: andres.navas@usach.cl

\end{footnotesize}

\end{document}